\documentclass[12pt,oneside,a4paper]{amsart}
\usepackage{microtype}
\pdfoutput=1
\usepackage{amssymb,amsmath,mathtools}
\usepackage{hyperref}

\usepackage[T1]{fontenc}
\usepackage{newtxtext,newtxmath}
\mathcode`\;="303B

\newcommand{\pr}{\mathcal{P}}

\newcommand{\dd}{\mathrm{d}}
\newcommand{\grad}{\mathrm{grad}}
\renewcommand{\div}{\mathrm{div}}

\newcommand{\dVolg}{\mathrm{dVol}_{g}}
\newcommand{\dV}{\dVolg}
\renewcommand{\epsilon}{\varepsilon}
\renewcommand{\phi}{\varphi}

\newtheorem{lemma}{Lemma}
\newtheorem{theorem}{Theorem}

\newtheorem{example}{Example}

\makeatletter
\renewcommand\subsection{\@startsection{subsection}{2}%
  \z@{.5\linespacing\@plus.7\linespacing}{.1\linespacing}%
  {\normalfont\scshape}}
\makeatother

\title{JKO schemes with general transport costs}
\author{Cale Rankin} 
\address{Fields Institute and Department of Mathematics, University of Toronto} 
\email{cale.rankin@utoronto.ca}

\author{Ting-Kam Leonard Wong} 
\address{Department of Statistical Sciences, University of Toronto} 
\email{tkl.wong@utoronto.ca}
\thanks{The research of Leonard Wong is partially supported by an NSERC Discovery Grant (RGPIN-2019-04419).}

\date{\today}

\begin{document}

\begin{abstract}
We modify the JKO scheme, which is a time discretization of Wasserstein gradient flows, by replacing the Wasserstein distance with more general transport costs on manifolds. We show when the cost function has a mixed Hessian which defines a Riemannian metric, our modified JKO scheme converges under suitable conditions to the corresponding Riemannian Fokker--Planck equation. Thus on a Riemannian manifold one may replace the (squared) Riemannian distance with any cost function which induces the metric. Of interest is when the Riemannian distance is computationally intractable, but a suitable cost has a simple analytic expression.  We consider the Fokker--Planck equation on compact submanifolds with the Neumann boundary condition and on complete Riemannian manifolds with a finite drift condition. As an application we consider Hessian manifolds, taking as a cost the Bregman divergence. 
\end{abstract}

\maketitle

\section{Introduction}
\label{sec:introduction}

In this paper we modify the JKO scheme, which was named after the seminal work of Jordan, Kinderlehrer and Otto \cite{JKO98}, so as to encompass optimal transport costs other than the Wasserstein distance. We begin by reviewing the JKO scheme and explaining why one would want to modify it. 

The JKO scheme is a physically meaningful, iterative scheme for approximating the solutions to dissipative partial differential equations which may be interpreted as gradient flows in the Wasserstein space. Taking for granted some standard terminology (see Table \ref{ta:notation} and Section \ref{sec:preliminaries}), the scheme for the heat equation is as follows: Let $(M,g)$ be a Riemannian manifold, $\mathcal{W}_2$ the (quadratic) Wasserstein distance on $\mathcal{P}(M)$, and $\rho_0$ a probability density with respect to the Riemannian volume measure. Fix $\tau>0$, set $\rho^\tau_0 = \rho_0$ and recursively define
\begin{equation}
  \label{eq:jko-def}\tag{JKO}
  \rho^\tau_{k+1} := \text{argmin}_{\rho \in \mathcal{P}(M)} \int_{M} \log \rho \, \dd \rho + \frac{1}{2\tau}\mathcal{W}^2_2(\rho,\rho^\tau_{k}),
\end{equation}
where $\int_M \log \rho \dd \rho$ is the entropy functional. Provided minimizers exist we may define the piecewise constant interpolation
\begin{equation}
  \label{eq:interpolant-def}
  \rho^\tau(t) := \rho^\tau_k \quad \text{for } t\in((k-1)\tau,k\tau],
\end{equation}
with $\rho^\tau(0) =\rho_0$. Jordan, Kinderlehrer, and Otto \cite{JKO98} proposed the update rule \eqref{eq:jko-def} in the case $M = \mathbf{R}^n$ and showed that as $\tau \rightarrow 0$ the curves $\rho^\tau$ converge weakly to the solution of the heat equation (they studied more generally the Fokker--Planck equation which amounts to adding a drift term (see \eqref{eqn:drift})  in \eqref{eq:jko-def}). Their result spurred a tremendous amount of related work, including gradient flows on metric spaces, probabilistic interpretations, as well as numerical implementations and applications. For further details see \cite{adams2011large, AGS08, P15, santambrogio2017euclidean} and the references therein.

The Wasserstein distance appearing in \eqref{eq:jko-def} is defined by
\[ \mathcal{W}_2^2(\rho,\rho^\tau_{k}) = \inf_{\pi \in \Pi(\rho,\rho^\tau_k)} \int_{M \times M} d^2(x,y) \, \dd \pi(x,y), \]
where $\Pi(\rho,\rho^\tau_k)$ is the set of Borel probability measures on $M \times M$ with marginals $\rho$ and $\rho^\tau_k$, and $d$ is the Riemannian distance on $(M, g)$. The Wasserstein distance is a natural distance on the space of probability measures \cite{V03, V08}. However, for many manifolds one does not have simple expressions for the Riemannian distance, making numerical computation of the Wasserstein distance and JKO step expensive. In this paper we show if $c(x,y)$ is any (possibly asymmetric) cost function on $M \times M$ such that in coordinates
\begin{equation}
  \label{eq:cg-def}
  g_{ij}(x) = -c_{x^i,y^j}(x,x),
\end{equation}
then one may replace the term $\frac{1}{2}\mathcal{W}_2^2(\rho, \rho_k^{\tau})$ in \eqref{eq:jko-def} by the transport cost
\[ \mathcal{T}_{c}(\rho,\rho^\tau_{k}) = \inf_{\pi \in \Pi(\rho,\rho^\tau_k)} \int_{M \times M} c(x,y) \, \dd \pi(x,y),\]
and obtain a sequence of interpolants which converge to the same limit as the original scheme. An alternate viewpoint of this result is when \eqref{eq:cg-def} defines a Riemannian metric, the JKO scheme with $\mathcal{T}_{c}$ in place of $\frac{1}{2}\mathcal{W}_2^2$ yields approximate solutions to the Riemannian heat or Fokker--Planck equation. While the limiting PDE is the same, the advantage occurs when, for a given metric $g$, one may find a cost $c$ realizing \eqref{eq:cg-def} for which $c(x,y)$ is significantly easier to compute than $d(x,y)$. %It can be shown that \eqref{eq:cg-def} implies the local approximation $c(x, y) \approx \frac{1}{2}d^2(x, y)$.

Such a case motivates our modification. As an application of our main result, stated in Theorem \ref{thm:main} below, we consider Hessian manifolds with a global chart, equivalently $M \subset \Omega$ where $\Omega$ is an open convex domain in $\mathbf{R}^n$ equipped with a Riemannian metric of the form $g_{ij}(x) = \phi_{ij}(x) := \varphi_{x^ix^j}(x)$ for $\phi$ a convex function. Letting $\varphi(x) = \frac{1}{2}|x|^2$ recovers the Euclidean metric. Hessian Riemannian metrics have been intensively studied in differential geometry \cite{shima2007geometry, SY97} and optimization theory \cite{alvarez2004hessian, wibisono2016variational}, yet there is no simple formula for the Riemannian distance except in special cases. However, the Bregman divergence
\begin{equation}
  \label{eq:breg-def}
  c(x,y) = B_{\phi}(x,y) := \phi(x)-\phi(y)-D\phi(y)\cdot(x-y),
\end{equation}
satisfies \eqref{eq:cg-def} and is straightforward to compute. Thus on such a Hessian manifold one may compute approximate solutions to the heat equation and Fokker--Planck equation using a JKO scheme with transport cost
\begin{equation}
  \label{eq:bw-div}
  \mathcal{B}_{\phi}(\rho,\rho^\tau_{k}) := \inf_{\pi \in \Pi(\rho,\rho^\tau_k)} \int_{\Omega \times \Omega} B_{\phi}(x,y) \, \dd \pi(x,y).
\end{equation}
The transport cost \eqref{eq:bw-div} was first considered by Carlier and Jimenez \cite{carlier2007monge} and, more recently, the authors \cite{RW23} studied the geometric structures on $\mathcal{P}(M)$ induced by \eqref{eq:bw-div}, which we call the Bregman--Wasserstein divergence.  We note that computation of \eqref{eq:bw-div} is, up to a coordinate transformation (see \cite[Proposition 3.2]{RW23} and \eqref{eqn:Bregman.mixed} below), equivalent to that of a Euclidean Wasserstein distance for which Brenier's theorem applies and many computational algorithms are now available. Since the Bregman divergence has numerous applications in statistics and applied mathematics (see for example \cite{A16, BMDG05, basseville2013divergence, AM03}), we believe the corresponding modified JKO scheme will be useful for studying related problems involving optimal transport. Regarding the use of general cost functions, we also note the recent work of L\'eger and Aubin-Frankowski \cite{leger2023} who developed a unifying perspective of gradient descent algorithms using generalized convex duality from optimal transport.

Before we provide background references, we state our main theorems. Our terminology and precise statements of our assumptions are provided in Section \ref{sec:preliminaries}. We will prove our main result for arbitrary costs on Riemannian manifolds. Then as a corollary we list our motivating case: the Bregman divergence on a subset of $\mathbf{R}^n$.
Let $(N,g)$ be an ambient Riemannian manifold with metric $g$, distance $d(x, y)$ and volume measure $\dVolg$, and let $M \subset N$ be an open submanifold. Assume $c:N \times N \rightarrow\mathbf{R}_+$ is a cost function satisfying \eqref{eq:cg-def}. Let $\rho_0$ be a probability density on $M$ (with respect to the Riemannian volume) which has finite entropy. Our modified JKO scheme for the Riemannian Fokker--Planck equation
\[ \partial_t\rho = \beta^{-1}  \Delta \rho + \div(\rho \nabla \psi),\]%\langle\nabla \psi, \nabla \rho\rangle, \]
where $\beta > 0$ is a constant and  $\psi: M \rightarrow \mathbf{R}_+$ is a smooth potential satisfying $\Vert \nabla \psi \Vert \leq C(1+\psi)$, is the recursion
\begin{equation}
  \label{eq:m-jko-def}
  \rho^\tau_{k+1} := \text{argmin}_{\rho \in \mathcal{P}(M)} \, \beta^{-1} \int_{M}  \log \rho \, \dd \rho + \int_{M} \psi \, \dd \rho + \frac{1}{\tau}\mathcal{T}_{c}(\rho,\rho^\tau_{k}).
\end{equation}
It will be shown that $\rho^\tau_{k+1}$ is well-defined as the unique minimizer. Let $\rho^\tau(t)$, $t \geq 0$, be the piecewise constant interpolants defined by \eqref{eq:interpolant-def}.

\begin{theorem}\label{thm:main}
Assume $c$ is a cost function on $N$ satisfying \eqref{eq:cg-def} as well as conditions A1, and A2. Assume further $N$ is $c$-convex with respect to itself and that there are constants $\lambda,\Lambda > 0$ such that $\lambda d^2(x,y) \leq c(x,y) \leq \Lambda d^2(x,y)$ for $x, y \in M$. Assume either:
  \begin{enumerate}
  \item[(i)] $M$ is pre-compact in $N$ with (non-empty) smooth boundary; or
    \item[(ii)] $M=N$ and $(M,g)$ is a complete Riemannian manifold with Ricci curvature bounded below.
    \end{enumerate}
     Then there exists a measurable function $\rho: [0,\infty) \times M \rightarrow\mathbf{R}_+$ such that for each $t \in [0,\infty)$
  \[ \rho^\tau(t) \rightharpoonup \rho(t) \text{ weakly in }L^1(M) := L^1(M; \dVolg),\]
  and $\rho$ solves the Fokker--Planck equation
  \begin{align*}
    \begin{cases}
      \partial_t\rho = \beta^{-1}  \Delta\rho + \div(\rho \nabla \psi) &\text{ in}\quad(0,\infty) \times M;\\
      \rho(0) = \rho_0 &\text{ on}\quad M.
    \end{cases}
  \end{align*}
Moreover, in case (i) $\rho$ is the unique classical solution satisfying the Neumann boundary condition, and in case (ii) $\rho$ is the unique classical solution with finite drift (meaning $\int_M \psi \dd \rho(t) < \infty$) and second moment. 
\end{theorem}

In the above $\Delta$ and $\div$ are the Laplace-Beltrami operator and Riemannian divergence. Table \ref{ta:notation} outlines our notation.  We list as a corollary the case of most interest to us; further details, and examples, of the Bregman case are given in Section \ref{sec:Bregman}.

\begin{theorem} \label{thm:main2}
Assume $N=\Omega$ be an open convex subset of $\mathbf{R}^n$. Assume $c(x,y) = B_{\phi}(x,y)$ is a Bregman divergence, where $\phi : \Omega \rightarrow \mathbf{R}$ is smooth convex function such that $D^2 \varphi$ is positive definite on $\Omega$ and $D\varphi(\Omega)$ is  convex. Assume either:
  \begin{enumerate}
  \item[(i)] $M$ is an open and precompact subset of $N$ with smooth boundary; or
  \item[(ii)] $M = N$ and $(M,g)$, where $g_{ij}(x) = D_{ij}\varphi(x)$ under the Euclidean coordinates on $\Omega$, is a complete Riemannian manifold with Ricci curvature bounded below and there are constants $\lambda,\Lambda > 0$ such that $\lambda d^2(x,y) \leq B_{\phi}(x,y) \leq \Lambda d^2(x,y)$ for $x, y \in \Omega$.
  \end{enumerate}
    Then the same conclusions as in Theorem \ref{thm:main} hold.

\end{theorem}

To conclude this introduction we provide a bit more history, and mention other works which have modified the JKO scheme.

Gradient flow schemes in the Wasserstein space were first studied by Otto in conjunction with pattern formation in magnetic fluids \cite{Otto98}. The connection to the Fokker--Planck equation was realized by Jordan, Kinderlehrer, and Otto \cite{JKO98}. Their work indicated there should be a differential structure for which the heat equation is a gradient flow. This structure was found by Otto \cite{O01} and was further developed by Ambrosio,  Gigli and Savar\'{e} \cite{AGS08} among other authors. All this relied on Brenier's characterization of optimal transport maps as gradients \cite{B91} and McCann's displacement interpolations \cite{M97,M94}.

After this numerous authors modified the JKO scheme to study different PDE by changing the functional to be minimized, but still using the Wasserstein distance  (Santambrogio's works \cite{santambrogio2017euclidean} and \cite[\S 8.4.2]{S15} contain excellent expositions).
In work more relevant to ours, some authors have considered JKO schemes with different transport costs. These include Agueh \cite{Agueh05} who considered transport costs 
\begin{equation}
  \label{eq:agueh-dist}
  \mathcal{T}_c(\mu,\nu) = \inf_{\mu,\nu \in \Pi(\mu,\nu)} \int_{\mathbf{R}^n \times \mathbf{R}^n} c(x-y)\, \dd\pi(x,y),
\end{equation}
where $c(\cdot)$ is strictly convex and Figalli, Gangbo, and Yolcu \cite{FGY11} who considered an iteration scheme for cost functions of the form $c(x,y) = \inf L(\gamma,\dot{\gamma})$,  
where the infimum is over differentiable curves $\gamma$ with endpoints $\gamma(0) = x$ and $\gamma(1) = y$ and $L$ is a Lagrangian. Natile, Peletier, and Savar\'e \cite{NPS11} considered contractivity properties of \eqref{eq:agueh-dist} along solutions of the Fokker--Planck equation. Figalli and Gigli \cite{FN10} modified the transport cost so as to consider the heat equation with Dirichlet boundary conditions.  Finally Zhang \cite{Zhang07}, Erbar \cite{Erbar10} and Savare \cite{savare07} considered the JKO scheme with the Wasserstein distance on Riemannian manifolds with Ricci curvature bounded below. Recently, Deb et al.~\cite{pal2023} introduced a notion of mirror Wasserstein gradient flow and characterized it as a scaling limit of the Sinkhorn algorithm.

The above works have informed ours, and have advantages and disadvantages over ours. Indeed, Agueh's work as well as Figalli, Gangbo's and Yolcu's consider a broader class of PDEs. On the other hand their work is restricted to Euclidean space whereas we consider the corresponding Riemannian equations. In addition, the transport costs in the above works are symmetric in $x$ and $y$ and thus do not include our motivating example \eqref{eq:breg-def}.

Our paper is structured as follows. In Section \ref{sec:preliminaries} we give preliminaries from optimal transport and real analysis. In Section \ref{sec:diff-form} we compute differentiation formulas for the entropy, drift and the transport cost $\mathcal{T}_c$. Then in Sections \ref{sec:bw-bounded} and \ref{sec:bw-unbounded} we prove Theorem \ref{thm:main}. Our proofs follow Figalli and Glaudo's book \cite{FG21} as well as Jordan, Kinderlehrer and Otto's paper \cite{JKO98}. We use details in the Riemannian setting from Zhang \cite{Zhang07} and Erbar \cite{Erbar10}. The Bregman case is discussed in Section \ref{sec:Bregman}. In Section \ref{sec:conclusion} we conclude and discuss some future directions. Auxiliary results about weak solutions of the Fokker--Planck equation are given in Appendix \ref{sec:weak-form-heat}.

\begin{table}[h!]
\begin{center}
\begin{tabular}{l|l}
  Symbol & Meaning\\ \hline
  %$\Omega$ & open convex subset of $\mathbf{R}^n$ \\
 % $\phi$ & convex function on $\Omega$ defining the Bregman divergence \\
  %$B_{\phi}(x,y)$ & Bregman divergence: $\phi(x)-\phi(y)-D\phi(y)\cdot(x-y)$ \\

  $(N, g = \langle \cdot , \cdot \rangle)$ & Ambient Riemannian manifold \\
  $d(x, y)$ & Riemannian distance on $N$ \\
  $M$ & Open submanifold of $N$ (either precompact or $N$) \\
 %  $g \text{ or } \langle\cdot,\cdot\rangle$ & Riemannian metric on $N$\\
   $\grad, \div$ & Riemannian gradient and divergence\\ 
   $D_i, D_{ij}$ & Euclidean partial derivatives (in coordinates)\\
   $u_{x^i}, \dot{u}$ & Partial derivative in coordinates and time derivative \\
%     $a \cdot b$ & Euclidean dot product \\
     $\pr(M)$ & Space of (Borel) probability measures on $M$\\
      $c(x,y)$ & cost function which induces $g$ \\
$\mathcal{T}_c, \mathcal{W}_2$ & Optimal transport cost w.r.t.~$c$ and Wasserstein distance w.r.t.~$d$
\\
% ${\bf B}$ & Bregman divergence on $M$\\
%   $\overline{\nabla}$ & Levi-Civita connection on $M$ \\
%   $\nabla, \nabla^*$ & Primal and dual connection on $M$ \\
% $\mathcal{P}^{\infty}(M)$ & Space of probability measures on $M$\\
% & with suitable regularity conditions\\
%   $\mathfrak{g} \text{ or } \llangle \cdot , \cdot \rrangle$ & Otto's Riemannian metric on $\mathcal{P}^{\infty}(M)$\\
%   $\overline{\nnabla}$ & Levi-Civita connection on $\mathcal{P}^{\infty}(M)$\\
%   $\nnabla, \nnabla^*$ & Primal and dual connections on $\mathcal{P}^{\infty}(M)$
\end{tabular}
\medskip
\end{center}
\caption{Notation used in the paper}\label{ta:notation}
\end{table}  

\section{Preliminaries}
\label{sec:preliminaries}

In this section we state the required definitions and assumptions and state some required preliminary results. For general background in differential geometry and optimal transport we refer the reader to \cite{AGS08, O83, V08}. Throughout we adopt the summation convention.

\subsection{Preliminaries on optimal transport and the JKO scheme}
\label{sec:prel-optim-transp}
Let $(N,g)$ be an ambient Riemannian manifold with metric $g$. We denote the Riemannian distance on $N$ by $d(x, y)$ and the volume measure by $\dVolg$. Let $M \subset N$ be an open submanifold of $N$. When we are interested in the Neumann problem we take $\overline{M}$ to be compact with (nonempty) smooth boundary; when we are interested in the global problem we take $M = N$.  
 As stated in the introduction, whilst one may regard the cost as inducing the Riemannian metric, it is equivalent to assume that a Riemannian metric is given and the cost agrees with it (in the sense of \eqref{eq:cg-def}). By a cost function on $N$ we mean $c\in C^\infty(N \times N)$ satisfying $c(x,y) \geq 0$ with equality if and only if $x = y$. In statistics (see, e.g., \cite{A16}) such a function is sometimes called a divergence, or generalized distance, on $N$.  We assume that $c$ is compatible with $g$ in the sense that \eqref{eq:cg-def} holds for any coordinate system (applied to both $x$ and $y$). Since $c$ vanishes on the diagonal of $N \times N$, we also have
\[
g_{ij}(x) = c_{x^i, x^j}(x, y)|_{x = y} = c_{y^i, y^j}(x, y)|_{x = y} = -c_{x^j, y^i}(x, y)|_{x = y}.
\]
It can be shown that \eqref{eq:cg-def} implies the local approximation $c(x, y) \approx \frac{1}{2}d^2(x, y)$. For additional details and results see \cite[Section 11.3]{CU14} and \cite{WY19b}. While $c(x, y) = \frac{1}{2}d^2(x, y)$ recovers the metric via \eqref{eq:cg-def} and is currently the default choice,\footnote{We note, however, that $d^2(x, y)$ is generally not globally smooth, so our conditions do not cover completely the Riemannian case as in \cite{Erbar10, Zhang07}. On the other hand, the existence of a smooth cost that induces a given metric was proved in \cite{M93}.} our approach, which is based on \eqref{eq:cg-def}, allows a wide variety of cost functions which may be more suitable in particular applications.

We denote the space of Borel probability measures on $M$ by $\mathcal{P}(M)$. We say that $\rho \in \mathcal{P}(M)$ has finite second moment if for some, and thus any, $x_0 \in M$ we have $\int_M d^2(x, x_0) \, \dd \rho(x) < \infty$. We are interested in Wasserstein gradient flows, and their approximation schemes, for the entropy functional
\begin{align*}
  E(\mu) =
  \begin{cases}
    \int_{M} \log \rho \, \dd \rho &\text{ if } \dd \mu(x) = \rho(x) \dVolg(x),\\
                          +\infty &\text{ otherwise}.
  \end{cases}
\end{align*}
The constant factor $\beta > 0$ in \eqref{eq:m-jko-def} refers to the inverse temperature and is immaterial in the mathematical analysis, so in the rest of the paper we fix once and for all that $\beta = 1$. We will identify an absolutely continuous probability measure with respect to $\dVolg$ with its density. Gradient flows of $E$ yield the heat equation, but it is little extra work to couple the entropy with a drift term thereby obtaining the Fokker--Planck equation. Thus we define the drift term
\begin{equation} \label{eqn:drift} D(\mu) = \int_M \psi \, \dd \mu,\end{equation}
for $\psi \in C^\infty(M)$ a fixed nonnegative function satisfying the estimate
\[ \|\nabla \psi(x)\| \leq C(\psi(x)+1). \]

Let $\mu,\nu \in \mathcal{P}(M)$. The optimal transport cost with cost $c$ is
\[ \mathcal{T}_{c}(\mu,\nu) = \inf_{\pi \in \Pi(\mu,\nu)}\int_{M \times M}c(x,y) \, \dd \pi(x,y),\]
where $\Pi(\mu,\nu)$ is the set of probability measures on $M \times M$ with marginals $\mu$ and $\nu$. When $c = B_{\phi}$ is a Bregman divergence $\mathcal{T}_{c}(\mu,\nu) = \mathcal{B}_{\phi}(\mu,\nu)$ becomes the Bregman--Wasserstein divergence \eqref{eq:bw-div}.

Let us now assume we have a probability measure $\rho_0 \in \mathcal{P}(M)$ with finite second moment  and entropy. In particular, $\rho_0$ is absolutely continuous. For each $\tau > 0$ we recursively define
\[ \rho^\tau_{k+1} := \text{argmin}_{\rho \in \mathcal{P}(M)} \, E(\rho)+D(\rho) + \frac{1}{\tau}\mathcal{T}_{c}(\rho,\rho^\tau_k).\]
It is useful to set
\begin{equation}
  \label{eq:jk-def}
  J_k(\rho) := E(\rho)+D(\rho) + \frac{1}{\tau}\mathcal{T}_{c}(\rho,\rho^\tau_k).
\end{equation}
Provided this sequence is well defined, which we prove, one may define a piecewise constant curve of (absolutely continuous) probability measures by
\begin{align*}
  \rho^\tau(t) =  \begin{cases}
    \rho_0 & \text{ for }t=0,\\
    \rho^\tau_k &\text{ for }t \in \big((k-1)\tau,k\tau\big].
  \end{cases}
\end{align*}
Our main results, stated already in Theorems \ref{thm:main} and \ref{thm:main2}, are that these interpolants converge to solutions of the Fokker--Planck equation. 
\subsection{Preliminary analysis}
\label{sec:backgr-weak-conv}

We will repeatedly consider sequences of probability densities and need to prove some subsequence converges in a suitable sense. We achieve this via a modification of the  Dunford--Pettis theorem. For completeness, we also recall some standard definitions and results. 

Let $(X, \mu)$ be a measure space. We say that a sequence $(f_n)_{n \geq 1} \subset L^1(X; \mu)$ converges weakly (in $L^1$) to $f \in L^1(X; \mu)$, written  $f_n \rightharpoonup f$, provided $\lim_{n \rightarrow \infty}\int_{X} f_n k \, \dd \mu = \int_{X} f k \, \dd \mu$ for all $k \in L^{\infty}(X; \mu)$. The Dunford-Pettis theorem (see e.g., \cite[Lemma 5.13]{K21}) states that when $\mu(X) < \infty$, a collection $\mathcal{A} \subset L^1(X; \mu)$ is relatively weakly compact if and only if it is uniformly integrable, i.e., $\lim_{R \rightarrow +\infty} \sup_{f \in \mathcal{A}} \int_{\{f_n \geq R\}}|f_n| \, \dd\mu = 0$. We also recall that $\mathcal{A}$ is uniformly integrable if and only if there exists a superlinear function $h: [0,\infty] \rightarrow [0,\infty]$ such that $\sup_{f \in \mathcal{A}} \int_X h(f_n) \, \dd \mu < \infty$. Next, let $X$ be a metric space and $(\rho_n)_{n \geq 1} \subset \mathcal{P}(X)$. We say that $\rho_n$ converges narrowly to $\rho \in \mathcal{P}(X)$ if $\lim_{n \rightarrow \infty}\int_{X} k \, \dd \rho_n = \int_{X} k \, \dd \rho$ for all bounded continuous functions $k$ on $X$.

\begin{lemma}\label{lem:mod-dp}
Let $(X, d)$ be a proper metric space, i.e., all closed balls are compact, and let $\mu$ be a Radon measure on $X$. Let $(f_n)_{n \geq 1} \subset L^1(X;\mu)$ be such that $f_n \geq 0$ and $f_n \, \dd \mu$ are probability measures, i.e., $\int_{X}f_n\, \dd \mu = 1$. Also suppose there exists a superlinear $h:[0,\infty) \rightarrow [0,\infty)$ and $x_0 \in X$ such that
 \begin{align*}
  \sup_{n} \int_X h(f_n) \, \dd\mu &< +\infty, 
  \text{ and } \sup_{n} \int_X d^2(x,x_0) f_n \, \dd\mu < +\infty.
 \end{align*}
 Then the sequence has a weakly convergent subsequence $f_{k_j} \rightharpoonup f$ such that $\int_{X}f \, \dd \mu = 1$ and, in addition, $f_{k_j} \, d \mu$ converges narrowly to $f \, d \mu$. 
  \end{lemma}
  \begin{proof}
  For completeness we provide a sketch of the proof. The superlinearity estimate and our assumptions on $(X, d)$ and $\mu$ allow us to apply Dunford--Pettis on each ball $B_N := B(x_0, N)$ where $x_0$ is arbitrary and $n \in \mathbf{N}$. Coupled with a diagonalization argument we obtain a limiting function $f \geq 0$ to which, up to a subsequence, $f_n \rightharpoonup f$ in $L^1(B_N)$ for every $N$. The second moment estimate implies the subsequence $f_n \, \dd \mu$ is tight as elements of $\mathcal{P}(X)$. Thus by Prokhorov's theorem (see \cite[Theorem 23.2]{K21b}) it converges narrowly to some probability measure. Testing on each $B_N$ shows that the narrow limit is $f \dd \mu$. Now it is standard to show that $f_n \rightharpoonup f$ in $L^1(X; \mu)$.
  \end{proof}

  \subsection{Lower semicontinuity and convexity of $J_k$}
\label{sec:lower-semic-conv}

As a combination of standard results we obtain the following lemma. 
\begin{lemma}\label{lem:jk-lsc}
Provided $\rho^{\tau}_k \in \mathcal{P}(M)$ satisfies $E(\rho^{\tau}_k) < \infty$ and $D(\rho^{\tau}_k) < \infty$, the functional $J_k$ on $\mathcal{P}(M)$ as defined in \eqref{eq:jk-def} is lower semicontinuous (with respect to narrow convergence) and strictly convex. 
\end{lemma}
\begin{proof}
  We provide references: The transport cost is lower semicontinuous and convex by \cite[Lemma 4.3 and Theorem 4.8]{V08}. The lower semicontinuity of the entropy and drift follows from \cite[Lemma 5.1.7]{AGS08} and \cite[Theorem 29.20]{V08}. Finally, the drift term is linear, and the entropy is strictly convex. 
\end{proof}

\subsection{On $c$-convex sets and $c$-segments}
\label{sec:c-convex-sets}

When working with general cost functions certain assumptions are standard \cite{MTW05,V08}. Namely, we assume:

\medskip

\noindent\textbf{A1. } For all $x \in N$ the map $y \mapsto -\nabla_{x}c(x,y)$ from $N$ to $T_{x}^*N$ is injective. Similarly, for all $y \in N$ the map $x \mapsto -\nabla_{y}c(x,y)$ from $N$ to $T_{x}^*N$ is injective.

\smallskip

\noindent
\textbf{A2. } For each $(x,y) \in N^2$ and choice of coordinate systems about $x$ and about $y$
\[ \det c_{x^i,y^j}(x,y) \neq 0.\]
The condition A2, unfortunately, implies the manifold $N$ is not compact \cite{BLMR23}.

A novel feature of our argument for dealing with general costs is estimates along $c$-segments. These are a class of curves, akin to line segments, introduced by Ma, Trudinger, and Wang \cite{MTW05} to study the regularity of optimal transport. In the Riemannian setting \cite[Definition 12.10]{V08} the $c$-segment with respect to $x$ joining $y_0$ to $y_1$ is the curve $(y_t)_{0 \leq t \leq 1}$ defined by
\begin{equation}
  \label{eq:c-seg}
  \nabla_{x}c(x,y_t) = (1-t)\nabla_{x}c(x,y_0) + t \nabla_{x}c(x,y_1). 
\end{equation}
When such a curve exists it is unique by A1 and is smooth by  A2. We also assume that for any $x, y_0, y_1 \in N$ there exists a $c$-segment with respect to $x$ joining $y_0$ to $y_1$. That is, $N$ is $c$-convex with respect to itself. 

\begin{lemma} [Initial velocity of $c$-segment] \label{lem:c-seg}
If $(y_t)_{0 \leq t \leq 1}$ is a $c$-segment with respect to $x$ from $x$ to $y$, then $\dot{y}_0 = -\nabla_{x}c(x,y)$.
\end{lemma}
\begin{proof}
  Since $z \mapsto c(x,z)$ is minimized at $x$, $\nabla_x c(x, x) = 0$ and \eqref{eq:c-seg} reduces to
\begin{equation} \label{eq:c-seq2}
\nabla_x c(x, y_t) = t \nabla_x c(x, y_1), \quad 0 \leq t \leq 1.
\end{equation}
In a coordinate neighbourhood $U$ about $x$ we have
\[
c_{x^i}(x,y_t) = t c_{x^i}(x,y).
\]
With $t$ so small that $y_t \in U$, differentiating with respect to $t$ implies
\begin{equation} \label{eq:c-seg-first}
c_{x^i, y^j}(x, y_t) \dot{y}_t^j = c_{x^i}(x, y).
\end{equation}
When $t = 0$ then $y_0 = x$,  and the identity $g_{ij}(x) = -c_{x^i, y^j}(x, x)$ implies
\begin{equation*}
\dot{y}_0^k = g^{ki}(x)c_{x^i,y^j}(x,x)\dot{y}^j_0 = -g^{ki}(x,x)c_{x^i}(x,y).
\end{equation*}
That is, $\dot{y}_0 = -\nabla_{x}c(x,y)$.
\end{proof}

\section{Differentiation formulas}
\label{sec:diff-form}
We collect in the following lemma some required differentiation formulas. We work under the setting of Section \ref{sec:prel-optim-transp}.

\begin{lemma}\label{lem:diff-identities}
  Assume either
  \begin{enumerate}
  \item[(i)] $\overline{M}$ is compact in $N$ with (non-empty) smooth boundary, $c $ is a cost function on $N$, and $\xi$ is a $C^{0,1}$ vector field on $\overline{M}$ tangential to $\partial M$, i.e., $\langle \xi,   \mathbf{n} \rangle = 0$ on $\partial M$ for $\mathbf{n}$ the outer unit normal; or
    \item[(ii)] $(M,g)$ is a complete Riemannian manifold, $c$ is a cost function on $M$, and $\xi$ is a $C^{0, 1}$ vector field with compact support.
    \end{enumerate}
    Then for $\Phi_t : M \rightarrow M$ the flow of $\xi$ (which is, for small time, a diffeomorphism), $\rho_0 \in \mathcal{P}(M)$ with finite entropy and second moment, $\rho_t := (\Phi_t)_{\#} \rho_0$, and $\mu \in \mathcal{P}(M)$ with $\mathcal{T}_c(\rho_0, \mu) < \infty$, the following differentiation formulas hold:
    \begin{align}
\label{eq:diff1}   \frac{\dd}{\dd t} \Big\vert_{t = 0} D(\rho_t) &=  \int_M \langle \nabla \psi, \xi \rangle \, \dd \rho_0,\\
\label{eq:diff2}   \frac{\dd}{\dd t} \Big\vert_{t = 0} E(\rho_t) &=  -\int_{M} \div \xi \, \dd\rho_0,\\
\label{eq:diff3}      \limsup_{t \rightarrow 0^+} \frac{1}{t}\left( \mathcal{T}_{c}(\rho_t,\mu) - \mathcal{T}_{c}(\rho_0,\mu)\right) &\leq \int_{M\times M} \langle\xi(x), \nabla_{x}c(x,y)\rangle \, \dd \pi(x,y), 
    \end{align}
    where $\pi \in \Pi(\rho_0, \mu)$ is optimal for the transport $\mathcal{T}_{c}(\rho_0,\mu)$. 
  \end{lemma}
  \begin{proof}
  Equation \eqref{eq:diff1} follows from the continuity equation and whilst \eqref{eq:diff2} appears in the literature \cite{Erbar10}, for completeness we include a proof. Consider
    \begin{equation}
      \label{eq:push-forward-cond}
      E(\rho_t) = \int_{M} ( \log \rho_t)\rho_t \dVolg = \int_{M} \log(\rho_t(\Phi_t(x))) \dd \rho_0.
    \end{equation}  
  The change of variables formula (\cite[pg. 12]{V08}) implies
    \begin{equation}
      \label{eq:cov}
      \rho_0(x) = \rho_t(\Phi_t(x))\mathcal{J}_t(x),
    \end{equation}
    for $\mathcal{J}_t$ the Jacobian determinant of $\Phi_t$.    Let $\{U_{\ell}\}$ be a finite covering of the support of $\xi$ with coordinate neighbourhoods. In each neighbourhood 
    \[ \mathcal{J}_t(x) = \lim_{\epsilon \rightarrow 0} \frac{\text{Vol}_g(\Phi_t(B_\epsilon(x)))}{\text{Vol}_g(B_\epsilon(x))} = \frac{\sqrt{\det g_{ij}(\Phi_t(x))} \, | \!\det D\Phi_t|}{\sqrt{\det g_{ij}(x)}}. \]
This gives 
  \begin{align*}
 \int_{U_{\ell}} \log(\rho_t(\Phi_t(x))) \dd \rho_0 = \int_{U_{\ell}} \log\left(\frac{\rho_0(x)\sqrt{\det g_{ij}(x)}}{\sqrt{\det g_{ij}(\Phi_t(x))} |\det D\Phi_t|}\right) \dd \rho_0.
  \end{align*}
Since $\xi \in C^{0, 1}$ differentiation under the integral is justified by the dominated convergence theorem. Thus
\begin{align*}
&\frac{\dd}{\dd t} \int_{U_{\ell}} \log(\rho_t(\Phi_t(x))) \rho_0 \dVolg \\
&= - \int_{U_{\ell}} \frac{1}{\sqrt{\det g_{ij}(\Phi_t)}} \dot{\Phi}_t^k D_k \sqrt{\det g_{ij}(\Phi_t)} + \text{tr}\big((D\Phi_t)^{-1}D\dot{\Phi_t}\big)  \dd \rho_0,
\end{align*}
  wehre $\text{tr}$ denotes the trace. Thus at time $0$, when $\Phi_0, D\Phi_0 = \mathrm{Id}$ and $\dot{\Phi}^k_0 = \xi^k$,
  \begin{align} \label{eq:ub-ent-diff}
-\int_{U_{\ell}} \frac{1}{\sqrt{\det g_{ij}}} D_k \Big( \xi^k \sqrt{\det g_{ij}}\Big) \dd \rho_0 = -\int_{U_\ell} \div \xi \, \dd \rho_0. 
  \end{align}
The last equality uses the coordinate expression of the Riemannian divergence as in \cite[Proposition 2.46]{L18}. From this we obtain \eqref{eq:diff2} using a partition of unity subordinate to our coordinate neighborhoods. 

    To prove \eqref{eq:diff3} let $\pi \in \Pi(\rho_0,\mu)$ be an optimal plan for $\mathcal{T}_c(\rho_0,\mu)$. Then $(\Phi_t, \mathrm{Id})_{\#}\pi \in \Pi(\rho_t,\mu)$ and so is a sub-optimal plan for $\mathcal{T}_c(\rho_t,\mu)$. Thus,
    \begin{align*}
      &\mathcal{T}_{c}(\rho_t,\mu) - \mathcal{T}_{c}(\rho_0,\mu)\\ &\leq  \int_{M \times M} c(x,y) \, \dd (\Phi_t,\mathrm{Id})_{\#}\pi(x,y) -  \int_{M \times M} c(x,y) \, \dd \pi(x,y) \\
      &= \int_{M \times M} c(\Phi_t(x),y) - c(x,y) \, \dd \pi(x,y).
    \end{align*}
    Applying Taylor's theorem to the function $t \mapsto c(\Phi_t(x),y)$ yields
    \[ \mathcal{T}_{c}(\rho_t,\mu) - \mathcal{T}_{c}(\rho_0,\mu) \leq \int_{M \times M}  t\langle \nabla_{x}c(x,y), \xi(x) \rangle + \frac{t^2}{2}R(\tau,x,y) \, \dd \pi(x,y).\]
Our assumption that either the space or the support of $\xi$ is compact along with $\xi \in C^{0,1}$ imply for $t > 0$ sufficiently small, $R(t,x, y) \leq C$ for $C > 0$ independent of $x$ and $y$. Thus, since $\pi$ is a probability measure, we obtain
    \[ \limsup_{t \rightarrow 0} \frac{1}{t}\left( \mathcal{T}_{c}(\rho_t,\mu) - \mathcal{T}_{c}(\rho_0,\mu)\right) \leq \int_M  \langle \nabla_{x}c(x,y), \xi(x) \rangle \, \dd \pi(x,y).   \]
This is \eqref{eq:diff3} and completes the proof. 
  \end{proof}

\section{Pre-compact submanifolds}
\label{sec:bw-bounded}
In this section we prove case (i) of Theorem \ref{thm:main}. There are three key steps: the existence of $\rho^\tau_{k+1}$ for $k=0,1,2,\dots$, a derivation of the minimality equation satisfied by $\rho^\tau_{k+1}$, and finally the derivation of the Fokker--Planck equation. In the following all hypotheses of Theorem \ref{thm:main}(i) are assumed.

\subsection{Existence of minimizer}
\label{sec:existence-minimizer-2}
Recall that $\rho_0$ is a density on $M$ with finite entropy.

\begin{lemma}\label{lem:well-posed}
For each $\tau > 0$ and $k = 0,1,2 \ldots$, the minimization problem \eqref{eq:m-jko-def} has a unique minimizer $\rho_{k+1}^\tau \in \mathcal{P}(M)$ which satisfies
  \begin{equation}
    \label{eq:ent-decr}
    E(\rho^\tau_{k+1})+D(\rho^\tau_{k+1}) \leq E(\rho^\tau_{k})+D(\rho^\tau_k),
  \end{equation}
  in particular each $\rho^\tau_k$ has finite entropy and drift. 
\end{lemma}
\begin{proof}
By induction we prove the existence of $\rho_{k+1}^{\tau}$ when $\rho^\tau_k$ is well defined. Recall $\rho^\tau_{k+1} = \text{argmin}_{\rho \in \mathcal{P}(M)}J_k(\rho)$, where
  \[ J_k(\rho)=\int_{M} \log \rho \, \dd \rho + \int_{M} \psi \, \dd \rho + \frac{1}{\tau}\mathcal{T}_{c}(\rho,\rho^\tau_{k}).\]
Since $M$ is precompact in $N$, $\text{Vol}_g(M)$ is finite. Since also $\rho \log \rho \geq -1$ and $\psi \geq 0$, $J_k$ is bounded below. Thus we may take a minimizing sequence $\{\rho^{\tau}_{k,i}\}_{i \geq 1}$ satisfying
  \[ \lim_{i \rightarrow \infty}J_k(\rho^{\tau,k}_{i}) = \inf_{\rho \in \mathcal{P}(M)}J_k(\rho). \]
We assume, without loss of generality, that $C = \sup_i J_k(\rho^{\tau}_{k,i}) < \infty$.
  
  We extract a weakly convergent subsequence via Lemma \ref{lem:mod-dp}. Indeed, since $M$ is bounded we have trivially $\sup_i\int_M d(x, x_0)^2 \dd \rho_{k, i}^{\tau} < \infty$. In addition, 
  \begin{align*}
    C \geq J_k(\rho^\tau_{k,i}) &= \int_{M} \psi \, \dd \rho^\tau_{k,i} + \int_{M}  \log \rho^\tau_{k,i} \, \dd \rho^\tau_{k,i} + \frac{1}{\tau}\mathcal{T}_{c}(\rho^\tau_{k,i},\rho^\tau_{k})\\
                       &\geq \int_M (\rho^\tau_{k,i} \log \rho^\tau_{k,i})_{+} \dVolg- \int_M (\rho^\tau_{k,i} \log \rho^\tau_{k,i})_{-} \dVolg\\
    &\geq  \int_M (\rho^\tau_{k,i} \log \rho^\tau_{k,i})_{+} -\text{Vol}_{g}(M). 
  \end{align*}
  Since $z \mapsto (z \log z)_+$ is superlinear, Lemma \ref{lem:mod-dp} applies. Thus up to a subsequence $\rho^\tau_{k,i}$ converges narrowly to some absolutely continuous probability measure $\rho^\tau_{k+1}$ which, via the the weak lower semicontinuity and strict convexity of $J_k$ (Lemma \ref{lem:jk-lsc}), is the unique minimizer.
  
  To conclude the proof, we note since $\rho^\tau_{k+1}$ minimizes $J_k$, \eqref{eq:ent-decr} is simply a restatement of $J_k(\rho^\tau_{k+1}) \leq J_k(\rho^\tau_k)$. 
  \end{proof}

\subsection{Minimality condition}
\label{sec:minimality-condition-1}

Using Lemma \ref{lem:diff-identities} it is straightforward to find the Euler--Lagrange equation satisfied by $\rho^\tau_{k+1}$. 

\begin{lemma}\label{lem:min-cond}
Consider $\rho^{\tau}_{k+1}$ defined by \eqref{eq:m-jko-def}. Then for any $C^{0,1}(\overline{M})$ vector field $\xi$ tangential to $\partial M$ and $\pi$ an optimal plan for $\mathcal{T}_{c}(\rho^\tau_{k+1},\rho^\tau_k)$, we have
  \begin{equation}
    \label{eq:min-cond}
    \int_{M}\div \xi  - \langle\nabla \psi, \xi\rangle\, \dd \rho^\tau_{k+1} = \frac{1}{\tau}\int_{M \times M }\langle   \xi(x), \nabla_{x}c(x,y) \rangle\, \dd\pi(x, y).
  \end{equation}
\end{lemma}
\begin{proof}
As in Lemma \ref{lem:diff-identities} we let $\Phi_t$ denote the flow of $\xi$, and set\footnote{Thus, in this proof, $\rho_0 := \rho^\tau_{k+1}$, and is not the initial condition.} $\rho_t = (\Phi_t)_{\#} \rho^\tau_{k+1}$. 
Minimality implies for $t > 0$,
\[ 0 \leq \frac{1}{t}\big(J_k(\rho_t) - J_k(\rho_0) \big).\]
Taking a limit supremum as $t \rightarrow 0^+$ and employing Lemma \ref{lem:diff-identities} yields
\[ 0 \leq - \int_{M} \div \xi \, \dd \rho_0 +  \int_M \langle \nabla \psi, \xi \rangle \, \dd \rho_0 +  \frac{1}{\tau}\int_{M\times M}\langle   \nabla_{x}c(x,y), \xi(x) \rangle \, \dd \pi(x,y). \]
Since the same argument applies with $-\xi$ in place of $\xi$, we obtain the opposite inequality and subsequently \eqref{eq:min-cond}.
\end{proof}

\subsection{Derivation of the Fokker--Planck equation}
\label{sec:deriv-heat-equat}
Now we may complete the proof of Case (i) of Theorem \ref{thm:main}.
\begin{proof}[Proof of Theorem \ref{thm:main}(i)]
 \textit{Step 1. (Weak convergence on $[0,T] \times M$)} As shown in Lemma \ref{lem:well-posed}, we have that $E(\rho^\tau_k)+D(\rho^\tau_k)$ is decreasing in $k$. The entropy is bounded below by $-\text{Vol}_{g}(M)$. Thus for any $T > 0$,
 \begin{equation*} 
 \int_0^T \int_{M} (\rho^\tau\log \rho ^\tau)_{+}  \, \dVolg \dd t  \leq T \left(E(\rho_0)+\text{Vol}(M)+\sup_M |\psi|\right) < \infty.
\end{equation*}
Lemma \ref{lem:mod-dp} and a diagonalization argument yield $\rho:[0,\infty)\times M \rightarrow \mathbf{R}$ and a subsequence, still denoted by $\rho^{\tau}$, such that $\rho^\tau$ converges to $\rho$ (narrowly and weakly) on $[0,T]\times M$ for every $T>0$.
 
 To strengthen this to weak convergence for every $t$, we use that $\rho$ is smooth (in the sense of Lemma \ref{lem:compact-heat-eqn}). This follows from what we now prove: that $\rho$ is the unique distributional solution of the Fokker--Planck equation. Uniqueness of the limit implies that convergence holds without taking subsequences.

  \medskip
  
  \textit{Step 2. (Derivation of distributional form of the Fokker--Planck equation)}  Let $\zeta \in C^\infty(\overline{M})$ satisfy the Neumann condition $\langle \nabla \zeta(x), {\bf n}(x)\rangle = 0$ on $\partial M$.

We aim to derive the discretized version of the Fokker--Planck equation (\eqref{eq:heat-eq-deriv1} below). Fix $x,y \in M$ and let $(y_t)_{0 \leq t \leq 1}$ be the $c$-segment with respect to $x$ joining $x$ to $y$, so that $\nabla_x c(x, y_t) = t \nabla_x c(x, y)$ by \eqref{eq:c-seq2}. Lemma \ref{lem:c-seg} implies $\dot{y}_0 = - \nabla_x c(x, y)$. Taylor's theorem applied to $h(t) = \zeta(y_t)$ yields
\begin{equation}
\label{eq:gc-diff-identity}
 \zeta(y) - \zeta(x) + \langle \nabla_x c(x, y), \nabla \zeta(x) \rangle = R(x, y) := \int_0^1(1-t)h''(t) \, \dd t.
\end{equation}

We claim $R$ satisfies an estimate
\begin{equation}
  \label{eq:r-goal}
  R(x, y) \leq K c(x,y),
\end{equation}
for a constant $K > 0$ independent of $x,y \in M$. In coordinates about $y_t$ we have
\begin{equation}
  \label{eq:hdd}
  h''(t) = \zeta_{kl}(y_t)\dot{y}^k_t\dot{y}^l_t +\zeta_{a}(y_t)  \ddot{y}^a_t.
\end{equation}
Returning to \eqref{eq:c-seg-first} and differentiating again (though now $x,y$ may be in separate coordinate neighbourhoods) we obtain
\begin{equation}
  \label{eq:c-seg-second}
  c_{x^i,y^jy^k}(x,y_t)\dot{y}^j_t\dot{y}^k_t + c_{x^i,y^j}(x,y_t)\ddot{y}^j_t = 0. 
\end{equation}
From \eqref{eq:hdd}, \eqref{eq:c-seg-second}, and condition A2 we have
\begin{align*}
  h''(t) &= \zeta_{kl}(y_t)\dot{y}^k_t\dot{y}^l_t -\zeta_{a}(y_t)c^{a,i}(x,y_t)c_{i,jk}(x,y_t)\dot{y}^j_t\dot{y}^k_t\\
  &\leq K \Vert \nabla_{x}c(x,y)\Vert^2,
\end{align*}
for a constant $K$ initially depending on the choices of coordinate neighbourhoods about $x$ and $y_t$. But since we may cover $M$ by finitely many coordinate neighbourhoods, by taking the maximum of the resulting constants $K$, and noting that the definition of $h(t)$ is coordinate independent, we obtain a finite constant $K$ such that 
\[
h''(t) \leq K \|\nabla_x c(x, y) \|^2.
\]
for all $x, y \in M$. Finally because $\nabla_{x}c(x,x) = 0$ and $y \mapsto \nabla_{x}c(x,y)$ is Lipschitz on $M$ (uniformly in $x$), $\Vert \nabla_{x}c(x,y) \Vert \leq K\, d(x,y)$. Our assumption $\lambda d(x,y)^2 \leq c(x,y)$ gives \eqref{eq:r-goal} and thus
\begin{align}
\label{eq:zeta-id}  |\zeta(y) - \zeta(x) + \langle  \nabla_{x}c(x,y), \nabla \zeta(x) \rangle| \leq K \, c(x,y).
\end{align}
Integrating against an optimal plan for the transport $\mathcal{T}_{c}(\rho^\tau_{k+1},\rho^\tau_k)$ implies
\begin{align*}
 \left\vert \int_{M \times M} \zeta(y) - \zeta(x) + \langle \nabla_{x}c(x,y) , \nabla  \zeta(x) \rangle \, \dd \pi(x,y)\right\vert \leq K \mathcal{T}_c(\rho^\tau_{k+1},\rho^\tau_k).
\end{align*}
Finally by the minimality condition, Lemma \ref{lem:min-cond}, applied with the tangential vector field $\xi = \nabla \zeta$ and the marginal condition we obtain
\begin{align}
\label{eq:heat-eq-deriv1}
  \left\vert \int_M \zeta \, \dd \rho^\tau_{k} - \int_M \zeta \, \dd \rho^\tau_{k+1} + \tau \int_M \Delta \zeta - \langle\nabla \psi, \nabla \zeta\rangle\, \dd \rho^\tau_{k+1}  \right\vert \leq K \mathcal{T}_c(\rho^\tau_{k+1},\rho^\tau_k) .
\end{align}

It is standard that the rest of conclusions of Theorem \ref{thm:main}(i) follow from \eqref{eq:heat-eq-deriv1} (see, for example, \cite{Zhang07}). For completeness we include the details. 

Take $ \eta  \in C^\infty_c([0,\infty))$ so $\text{spt }\eta \subset [0,T)$ for some $T$. We simultaneously multiply \eqref{eq:heat-eq-deriv1} by $ \eta (k\tau)$ and sum from $k=0$ to $k_0$ where $k_0\tau \geq T$ and  obtain
\begin{align}
 \nonumber \Big|
\sum_{k=0}^{k_0} \,  \eta (k\tau)\left(\int_{M} \zeta \, \dd  \rho^\tau_{k} -\int_{M} \zeta \, \dd  \rho^\tau_{k+1}\right) +  &\int_{M}( \Delta \zeta -  \langle\nabla \psi,\nabla \zeta\rangle) \tau \eta (k\tau) \,  \dd \rho^\tau_{k+1}
  \Big|\\
 \label{eq:heat-eq-deriv2} &\leq  \tau C,
\end{align}
where $C = C(\zeta,c,\rho_0, \eta )$ is independent of $\tau$. Here we have used that $\sum_k \mathcal{T}_c(\rho^\tau_{k+1},\rho^\tau_k)$ is bounded by a telescoping sum; we defer the details to \eqref{eq:sum-bw}. 
We proceed to compute the terms in \eqref{eq:heat-eq-deriv2}. First,
\begin{align*}
 \sum_{k=0}^{k_0}  \int_{M} \tau  \eta (k\tau) \Delta \zeta  \,  \dd \rho^\tau_{k+1} = \sum_{k=0}^{k_0}\int_{M}\Delta \zeta \int_{k\tau}^{(k+1)\tau} \eta (k\tau) \, \dd t \,  \dd \rho^\tau_{k+1}.
\end{align*}
If $t \in [k\tau,(k+1)\tau]$ we have $| \eta (k\tau) -  \eta (t)| \leq C( \eta )\tau $ whereby
\begin{align*}
 \sum_{k=0}^{k_0}  \int_{M} \tau  \eta (k\tau) \Delta \zeta  \,  \dd \rho^\tau_{k+1} = \sum_{k=0}^{k_0}\int_{M}\Delta \zeta \int_{k\tau}^{(k+1)\tau} \eta (t) \, \dd t \,  \dd \rho^\tau_{k+1} + O(\tau).
\end{align*}
Where $O(\tau)$ denotes a quantity satisfying $|O(\tau)| \leq C( \eta ,\zeta)\tau$. 
In addition, since $\rho^\tau(t) = \rho^\tau_{k+1}$ on $(k\tau,(k+1)\tau]$ we obtain
\begin{equation}
  \label{eq:to-sub1}
   \sum_{k=0}^{k_0}  \int_{M}  \tau \eta (k\tau)\Delta \zeta   \,  \dd \rho^\tau_{k+1} = \int_0^{\infty}\!\!\!\int_{M}  \eta (t)\Delta \zeta \rho^\tau \, \dVolg \, \dd t + O(\tau). 
 \end{equation}
 Similarly,
   \begin{align*}
   \sum_{k=0}^{k_0}  \int_{M}\tau  \eta (k\tau) \langle \nabla \psi , \nabla \zeta \rangle  \,  \dd \rho^\tau_{k+1} = \int_0^{\infty}\!\!\!\int_{M}  \eta (t) \langle \nabla \psi , \nabla \zeta \rangle \rho^\tau \, \dVolg \, \dd t + O(\tau). 
 \end{align*}
 
Consider the remaining terms in \eqref{eq:heat-eq-deriv2}. We compute
\begin{align*}
  \sum_{k=0}^{k_0}  \eta (k\tau)&\left(\int_{M} \zeta \, \dd  \rho^\tau_{k} -\int_{M} \zeta \, \dd  \rho^\tau_{k+1}\right)  \\
  &= \int_{M} \zeta  \eta (0) \rho_0 \, \dVolg  +\sum_{k=1}^{k_0}\int_{M} \zeta ( \eta (k\tau) -  \eta ((k-1)\tau)) \, \dd \rho^\tau_{k}\\
  &= \int_{M} \zeta  \eta (0) \rho_0 \, \dVolg + \sum_{k=1}^{k_0} \int_{M}\zeta \int_{(k-1)\tau}^{k\tau}  \eta _t(t) \, \dd t \, \dd \rho^\tau_{k}.
\end{align*}
(Here $\eta_t(t) = \eta'(t)$.) Then using again $\rho^\tau = \rho^\tau_k$ on $((k-1)\tau,k\tau]$ we have
\begin{align}
\nonumber
  \sum_{k=0}^{k_0}  \eta (k\tau)&\left(\int_{M} \zeta \, \dd  \rho^\tau_{k} -\int_{M} \zeta \, \dd  \rho^\tau_{k+1}\right)  \\
 \label{eq:to-sub2} &= \int_{M} \zeta  \eta (0) \rho_0 \, \dVolg + \int_{0}^{\infty}\int_{M}  \eta _t(t) \zeta \rho^\tau \, \dd t \, \dVolg.
\end{align}
Substituting \eqref{eq:to-sub1} and \eqref{eq:to-sub2} into \eqref{eq:heat-eq-deriv2} we obtain
\[
  \left|
     \int_{M} \zeta  \eta (0) \rho_0 \, \dVolg + \int_{0}^\infty\int_{M} (\zeta\partial_t \eta  +  \eta (\Delta \zeta - \langle\nabla \psi,\nabla \zeta\rangle))\rho^\tau \, \dVolg \dd t
   \right| \leq C\tau. \]
 On sending $\tau \rightarrow 0$ and using the weak $L^1_{\text{loc}}$ convergence we obtain the distributional form of the Fokker--Planck equation. Lemma \ref{lem:compact-heat-eqn} implies $\rho$ is the unique classical solution to the Fokker--Planck equation with Neumann boundary condition and initial condition $\rho_0$. 

 \medskip

\textit{Step 3. Weak $L^1$ convergence for every $t$.} Following \cite{JKO98} we strengthen the convergence to weak $L^1$ convergence for each $t$, that is for all $\zeta \in L^{\infty}(M)$
\begin{equation} \label{eq:stronger} \lim_{\tau \rightarrow0 }\int_{M} \zeta \, \dd \rho^\tau(t) = \int_{M} \zeta \, \dd  \rho(t).
\end{equation}
Indeed, it suffices to prove this for $\zeta \in C_c^{\infty}(M)$. Fix $\delta > 0$ and compute
  \begin{align}
  \nonumber & \Big\vert \int_M \zeta(x) \, \dd  \rho^\tau(t) -\int_M \zeta(x) \, \dd \rho(t )\Big\vert   \\
           \label{eq:t1}               &\leq \Big\vert \int_M \zeta(x) \, \dd \rho^\tau(t) - \frac{1}{2\delta}\int_{t-\delta}^{t+\delta}\int_M \zeta(x) \, \dd \rho^\tau(r) \, \dd r  \Big\vert \\
  \nonumber                                                  &\qquad+\Big\vert \frac{1}{2\delta}\int_{t-\delta}^{t+\delta}\int_M \zeta(x) \, \dd \rho^\tau(r) \, \dd r - \frac{1}{2\delta}\int_{t-\delta}^{t+\delta}\int_M \zeta(x) \, \dd \rho(r) \, \dd r \Big\vert \\
\nonumber   &\qquad\qquad +\Big\vert \frac{1}{2\delta}\int_{t-\delta}^{t+\delta}\int_M \zeta(x) \, \dd \rho(r) \, \dd r - \int_M \zeta(x) \, \dd \rho(t) \Big\vert. 
  \end{align}
  We will send $\tau \rightarrow 0$ then $\delta \rightarrow 0$ to show that the limit is zero. Note the last term has limit $0$ as $\delta \rightarrow 0$ because the solution of the Fokker--Planck equation is smooth. Namely the time derivative of $\rho$ is bounded on $M$\footnote{Moreover, when we use these calculations in the case $M$ is a complete manifold the time derivative will be bounded on support $\zeta$. } On the other hand the second term tends to $0$ as $\tau \rightarrow 0$ by the weak convergence. 
  
 We bound the first term in \eqref{eq:t1} by $C\sqrt{\delta+\tau}$ as follows. Using minimality, a telescoping sum, and estimates for the drift and entropy we obtain
   \begin{align}
  \label{eq:sum-bw} \frac{1}{\tau}\sum_{k=0}^\infty\mathcal{T}_c (\rho^\tau_{k+1},\rho^\tau_{k}) &\leq (D(\rho^\tau_0)+E(\rho^\tau_0))-(D(\rho^\tau_{k+1})+E(\rho^\tau_{k+1})) .
 \end{align}
Thus $\sum_{k=0}^\infty\mathcal{T}_c (\rho^\tau_{k+1},\rho^\tau_{k}) \leq C\tau.$  Assume $T \gg t$ is fixed and $N,N' \in \mathbf{N}$ satisfy $N\tau,N'\tau \leq T$.  We claim
 \begin{equation}
   \label{eq:b-bound}
   \mathcal{T}_c(\rho^\tau_{N'},\rho^\tau_N) \leq C\tau|N'-N|.
 \end{equation}
 Indeed, we invoke the triangle inequality by passing through the Wasserstein distance and obtain (assuming $N'>N$)
 \begin{align*}
   \mathcal{T}_c(\rho^\tau_{N'},\rho^\tau_N) &\leq C \mathcal{W}_2^2(\rho^\tau_{N'},\rho^\tau_N)\\
   &\leq C\left(\sum_{k=N}^{N'-1}\mathcal{W}_2(\rho^\tau_{k+1},\rho^\tau_{k})\right)^2\\
                         &\leq C (N'-N)\sum_{k=N}^{N'-1}\mathcal{W}_2^2(\rho^\tau_{k+1},\rho^\tau_{k})\\
   &\leq C(N'-N)\sum_{k=N}^{N'-1}\mathcal{T}_{c}(\rho^\tau_{k+1},\rho^\tau_{k}).
 \end{align*}
 The telescoping sum \eqref{eq:sum-bw} yields \eqref{eq:b-bound}.

 In addition for arbitrary probability densities $\rho',\rho$ we claim
 \begin{equation}
   \label{eq:b-test-est}
   \left\vert \int_M \zeta \, \dd \rho'  - \int_M \zeta \, \dd \rho \right\vert^2 \leq C(|\sup D\zeta|) \mathcal{T}_{c}(\rho,\rho'). 
 \end{equation}
Indeed if  $\pi$ is an optimal transport plan for $\mathcal{T}_{c}(\rho,\rho')$, we have
 \begin{equation*}
 \begin{split}
  &\left\vert
   \int_M \zeta \, \dd \rho'  - \int_M \zeta \, \dd \rho
   \right\vert^2 = \left\vert \int_M \zeta(y) - \zeta(x) \, \dd  \pi(x,y) \right\vert^2 \\
   &\leq \sup |D\zeta|^2  \frac{1}{\lambda} \int_M c(x,y) \dd \pi(x,y), %thanks for catching this
  \end{split}
\end{equation*}
 this last term is bounded by $\mathcal{T}_{c}( \rho,\rho')$.
 Thus for any $t,t' \leq T$ we may choose $N,N'$ such that $t \in ((N-1)\tau,N\tau]$ and similarly for $N'$. Then by \eqref{eq:b-bound}
 \[ \mathcal{T}_{c}(\rho^\tau(t),\rho^\tau(t')) =\mathcal{T}_{c}(\rho^\tau_{N'},\rho^\tau_N)  \leq C\tau|N-N'|.\]
 Which combined with \eqref{eq:b-test-est} implies
 \begin{align}
\nonumber   \left\vert \int_M \zeta \, \dd  \rho^\tau(t') - \int_M \zeta \, \dd \zeta\rho^\tau(t)\right\vert^2 &\leq C(\sup|D\zeta|)\tau|N-N'|\\
   \label{eq:b-holder} &\leq C(\sup|D\zeta|)(|t-t'|+\tau).
 \end{align}
 The latter inequality accounts for the fact that $t,t'$ may lie inside the intervals $(Nh,(N+1)\tau]$, $(N'h,(N'+1)\tau]$. Thus using \eqref{eq:b-holder} the term \eqref{eq:t1} is bounded by $C\sqrt{\tau+\delta}$. This establishes \eqref{eq:stronger} and completes the proof. 
 \end{proof}

\section{Complete manifold}
\label{sec:bw-unbounded}

In this section we remove the compactness requirements (and the Neumann boundary condition) of the previous section and prove Case (ii) of Theorem \ref{thm:main}. For this case $M = N$ and $(M,g)$ is a complete Riemannian manifold with $\lambda d^2 \leq c \leq \Lambda d^2$ and Ricci curvature bounded below. By the Bishop--Gromov inequality (which we take from Erbar's paper \cite[Proof of Lemma 4.1]{Erbar10}), there is $K$ depending on the lower bound for the Ricci curvature satisfying
\begin{equation}
  \label{eq:bg}
  \int_{M}e^{ -d(x,x_0)/2} \dVolg(x) \leq K < +\infty,
\end{equation}
where here and throughout $x_0$ denotes a fixed, but arbitrary, element of $M$. Inequality \eqref{eq:bg} is the key way in which the condition of Ricci curvature bounded below enters the proof. In this section we largely follow the original paper of Jordan, Kinderlehrer, and Otto \cite{JKO98}.

\subsection{Existence of a minimizer}
\label{sec:existence-minimizer}
This step differs most from the compact case. The entropy is not bounded below and thus, apriori, $J_k$ (which we wish to minimize) may not be bounded below. To overcome this, we establish that $\mathcal{T}_c$ dominates the negative part of the entropy. We let $M(\rho) = \int_{M} d(x,x_0)^2 \, \dd  \rho$ be the second moment of $\rho$ (about $x_0$) and establish the following estimates. 

\begin{lemma}
  Under the assumptions of Theorem \ref{thm:main} Case (ii) and with $\rho_0,\rho_1$  any nonnegative probability densities we have the following estimates.
  \begin{enumerate}
  \item[(i)] For all $\epsilon > 0$ and $M' \subset M$ there holds
    \begin{align}
      \label{eq:ent-lower} \int_{M'} (\rho \log \rho)_{-} \leq \int_{M'} e^{-d(x,x_0)/2} \, \dVolg + \epsilon M(\rho) + \frac{1}{4\epsilon} \int_{M'} \rho \, \dVolg.
    \end{align}
  \item[(ii)] The optimal transport cost satisfies 
    \begin{equation}
      \label{eq:moment-comparison}
      \mathcal{T}_c(\rho_0,\rho_1) \geq \frac{\lambda}{2} M(\rho_0) - \lambda M(\rho_1),
    \end{equation}
    where $\lambda> 0$ is from the statement of Theorem \ref{thm:main}.
  \item[(iii)] There holds
    \begin{equation}
      \label{eq:jk-bound-mom}
      J_k(\rho) \geq \frac{\lambda}{4\tau}M(\rho) - \frac{\lambda}{\tau} M(\rho^\tau_k)-C(K,\lambda,\tau).
    \end{equation}
  \end{enumerate}
\end{lemma}
\begin{proof}
  To prove \eqref{eq:ent-lower} note if $z \in (0,1]$ then $(z \log z)_{-} = z |\log z| \leq z^{1/2}$. Let
  \begin{align*}
    M_0 &= \{x \in M'; \rho (x) \leq e^{-d(x,x_0)}\text{ and }\rho(x) \leq 1\},\\
  \text{ and } M_1 &= \{x \in M' ; e^{-d(x,x_0)} \leq \rho(x)\text{ and }\rho(x) \leq 1\}.
  \end{align*}
  Then
  \begin{align*}
    \int_{M'} (\rho \log \rho)_{-}\dVolg  &= \int_{M_0}(\rho \log \rho)_{-}\dVolg + \int_{M_1} (\rho \log \rho)_{-}\dVolg\\
                                 & \leq \int_{M_0} \rho^{1/2} \dVolg + \int_{M_1} \rho d(x,x_0) \dVolg\\
    &\leq  \int_{M'} e^{-d(x,x_0)/2} \dVolg + \int_{M'} \epsilon d(x,x_0)^2 \rho +\frac{\rho}{4\epsilon} \dVolg.
  \end{align*}
  We've used $ab \leq \epsilon a^2 + b^2/4\epsilon$ with $a = d(x,x_0)$, $b=1$ and this proves \eqref{eq:ent-lower}.

  To prove \eqref{eq:moment-comparison} take $\rho_0,\rho_1 \in \mathcal{P}(M)$ and let $\pi$ be an optimal plan for $\mathcal{T}_c(\rho_0,\rho_1)$.
  Then by the condition $\lambda d^2(x,y) \leq c(x,y)$
  \begin{equation}
    \label{eq:to-simp}
    \mathcal{T}_c(\rho_0,\rho_1)  = \int_{M}c(x,y) \dd \pi(x,y) \geq \lambda \int_{M}d^2(x,y) \dd \pi(x,y).
  \end{equation}
  The triangle inequality (in the form $d(x,x_0) \leq d(x,y)+d(y,x_0))$ implies
  \[ d(x,y)^2 \geq \frac{1}{2}d(x,x_0)^2 - d(x_0,y)^2,\]
  whereby the marginal condition for $\pi$ implies \eqref{eq:to-simp} becomes
  \[\mathcal{T}_c(\rho_0,\rho_1) \geq \frac{\lambda}{2}M(\rho_0) - \lambda M(\rho_1). \]

  For \eqref{eq:jk-bound-mom} we use $\psi \geq 0$, \eqref{eq:ent-lower} with $\epsilon = \lambda/(4\tau)$, and \eqref{eq:moment-comparison} to obtain
  \begin{align*}
    J_k(\rho) &= \frac{1}{\tau}\mathcal{T}_c(\rho,\rho^\tau_{k}) + \int_{M} \psi \, \dd \rho + \int_{M} \log \rho \, \dd \rho\\
           &\geq \frac{1}{\tau}\mathcal{T}_c(\rho,\rho^\tau_{k}) - \int_{M} (\rho\log \rho)_{-} \dVolg\\
           &\geq \frac{\lambda}{2\tau}M(\rho) - \frac{\lambda}{\tau} M(\rho^\tau_k) - \int_{M} e^{-d(x,x_0)/2} \, \dVolg \\
           &\quad\quad - \frac{\lambda}{4\tau} M(\rho) - \frac{\tau}{\lambda} \int_{M} \rho \, \dVolg\\
    &\geq \frac{\lambda}{4\tau}M(\rho)- \frac{\lambda}{\tau} M(\rho^\tau_k) - C(K,\lambda,\tau).
  \end{align*}
  This is \eqref{eq:jk-bound-mom} and completes the proof. 
\end{proof}

\begin{lemma}\label{lem:bw-ub-exist}
  Under the assumptions of Theorem \ref{thm:main} (ii) and for each $k=1,2,\dots$ the minimization problem \eqref{eq:m-jko-def} has a unique minimizer $\rho^\tau_{k+1}$. \end{lemma}
\begin{proof}
 It is immediate from \eqref{eq:jk-bound-mom} that $J_k$ is bounded below (in this proof $\tau,k$ are fixed).   Because $J_k$ is bounded below we may take a minimizing sequence of probability densities $\rho^\tau_{k,i}$ satisfying
  \begin{equation}
    \label{eq:min-seqn}
    \lim_{i\rightarrow \infty}J_k(\rho^\tau_{k,i})  = \text{inf}_{\rho \in \mathcal{P}(M)}J_k(\rho).
  \end{equation}
  Without loss of generality $\sup_{i}J_k(\rho^\tau_{k,i}) < \infty$. To obtain a convergent subsequence by Lemma \ref{lem:mod-dp} it suffices to establish
  \begin{equation}
    \label{eq:conv-bound}
    \sup_{i}\int(\rho^\tau_{k,i}\log \rho^\tau_{k,i})_+ \dVolg < +\infty
  \end{equation}
  and a second moment bound.
  To this end, note from \eqref{eq:jk-bound-mom} we have
  \begin{equation}
    \label{eq:seqn-mom-bound}
    \sup_{i}M(\rho^\tau_{k,i}) < +\infty.
  \end{equation}
  Combined with
  \[J_k(\rho) \geq \frac{1}{\tau}\mathcal{T}_c(\rho,\rho^\tau_k)+\int(\rho\log \rho)_{+}\dVolg- \int(\rho\log \rho)_{-} \dVolg, \]
  into which we substitute \eqref{eq:moment-comparison} and \eqref{eq:ent-lower}, we obtain \eqref{eq:conv-bound}. We conclude the existence of a subsequence $\rho^\tau_{k,i_j}$ converging weakly to some $\rho^\tau_{k+1}$.

  Finally, it is again immediate by the lower semicontinuity and strict convexity of $J_k$ (Lemma \ref{lem:jk-lsc}) that $\rho^\tau_{k+1}$ uniquely minimizes $J_k$.
\end{proof}

The minimality equation for $\rho^{\tau}_{k+1}$ is unchanged once we take compactly supported $\zeta$ and, once more, follows immediately from Lemma \ref{lem:diff-identities}.  

\begin{lemma}\label{lem:ub-el}
  Let $\rho^\tau_{k+1}$ be the minimizer of the functional $J_k$. Then for every compactly supported Lipschitz vector field $\xi$ there holds
  \begin{equation}
    \label{eq:ub-el}
    \int_{M} \div \xi - \langle \nabla \psi , \xi \rangle\, \dd \rho^\tau_{k+1}  =  \frac{1}{\tau}\int_{M \times M} \langle\xi(x), \nabla_{x}c(x,y) \rangle \dd\pi(x,y),
  \end{equation}
where $\pi$ is an optimal plan for $\mathcal{T}_{c}(\rho^\tau_{k+1},\rho^\tau_{k})$. 
\end{lemma}

\subsection{Derivation of the Fokker--Planck equation}
\label{sec:deriv-heat-equat-1}

\begin{proof}[Proof. (Theorem \ref{thm:main}, Case (ii))]
\textit{Step 1. (Derivation of time discretized equation)}
We explain the modifications needed to obtain \eqref{eq:heat-eq-deriv1} and the weak $L^1$ convergence on $[0,\infty) \times M$. With these in hand the derivation of the Fokker--Planck equation and improvement to weak $L^1$ convergence on $M$ for each $t$ follows verbatim Case (i). Thus, let $\zeta\in C^\infty_c(M)$ be a test function. 

  The compact support of $\zeta$  implies that by arguing as in \eqref{eq:zeta-id} for all $x,y \in \text{spt }\zeta $ there holds 
  \begin{align}
  |\zeta(y) - \zeta(x) - \langle\nabla \zeta(x), \nabla_xc(x,y)\rangle|   \leq C(\zeta)d(x,y)^2. \label{eq:ub-interpolant}
  \end{align}
  This inequality clearly holds when neither $x$ nor $y$ is in $\text{spt }\zeta$. When $x$ is in the support of $\zeta$ but $y$ is not, we argue as follows. The function $x \mapsto \zeta(x) - \langle\nabla \zeta(x), \nabla_xc(x,y)\rangle=:h(x)$ is smooth, equal to $0$ outside the support of $\zeta$ and has a derivative bound depending on $c$ in the support of $\zeta$ and also on $\zeta$. Let $\gamma_t$ be the (Riemannian) geodesic with $\gamma_0 = y$ and $\gamma_1 = x$ and let $T$ be the supremum of $t$ such that $\gamma_t$ is not in $\text{spt}\zeta$. Then, noting $h(\gamma_T) = 0$, a standard estimate using a Taylor series yields
  \begin{align*}
    |\zeta(x) - &\langle\nabla \zeta(x), \nabla_xc(x,y)\rangle| \\
    &= |\zeta(x) - \langle\nabla \zeta(x), \nabla_xc(x,y)\rangle -  \zeta(\gamma_T) - \langle\nabla \zeta(\gamma_T), \nabla_xc(\gamma_T,y)\rangle|\\
    &\leq C(\zeta,\phi)d(x,\gamma_T)^2 \leq C d(x,y)^2. 
  \end{align*}
  Here we are using that $\frac{\dd^2}{\dd t^2}h(\gamma_t) = \langle \text{Hess }f \dot{\gamma}_t,\dot{\gamma}_t \rangle$ and that $\gamma_t$ is a constant speed geodesic.  A similar, but simpler, estimate yields \eqref{eq:ub-interpolant} for $y \in \text{spt }\zeta$ and $x \not\in \text{spt }\zeta$. Thus \eqref{eq:ub-interpolant} holds for all $x,y \in M$. 

We let $\pi$ denote an optimal plan for $\mathcal{T}_c(\rho^\tau_{k+1},\rho^\tau_k)$ and integrate \eqref{eq:ub-interpolant} with respect to $\pi$. Using $\lambda d(x,y)^2 \leq c(x,y) $, we have
   \[
   \left|
     \int_{M\times M} \zeta(y) - \zeta(x) -\langle\nabla \zeta(x), \nabla_xc(x,y)\rangle \, \dd  \pi(x,y)
   \right| \leq C \, \mathcal{T}_c(\rho^\tau_{k+1},\rho^\tau_k).\]
Combined with the minimality equation (Lemma \ref{lem:ub-el} with  $\xi = \nabla \zeta$) and the marginal condition we obtain
\begin{equation}
  \label{eq:bw-ub-fd}
  \begin{split}
  &\left|
    \int_{M} \zeta \, \dd  \rho^\tau_{k} - \int_{M} \zeta  \, \dd  \rho^\tau_{k+1} +\tau \int_{M} \Delta \zeta - \langle \nabla \psi , \nabla \zeta \rangle\, \dd \rho^\tau_{k+1}
  \right|\\
  &\leq C \, \mathcal{T}_c(\rho^\tau_{k+1},\rho^\tau_k).
  \end{split}
\end{equation}

The remainder of the proof concerns the convergence as $\tau \rightarrow 0.$

\textit{Step 2: Convergence estimates.}
Now we establish the necessary estimates to apply Lemma \ref{lem:mod-dp} and conclude the existence of a weakly/narrowly convergent subsequence of $\rho^\tau$ on $[0,\infty)\times M$. Fix $T>0$, we show there is $C$ depending on $T$ such that whenever $k_0\tau \leq T$ there holds
 \begin{align}
   \label{eq:ub-mom-bound} M(\rho^\tau_k) &\leq C,\\
   \label{eq:ub-pos-bound} \int (\rho^\tau_k\log \rho^\tau_k)_{+} \dVolg &\leq C.
 \end{align}

 We begin by showing $J_k(\rho_{k+1})$ is nonincreasing in $k$, and computing a telescoping sum we'll use repeatedly. First, the minimality of $\rho^\tau_{k+1}$ implies
 \begin{align}
   \label{eq:7}
   \frac{1}{\tau}&\mathcal{T}_c(\rho^\tau_{k+1},\rho^\tau_{k}) + E(\rho^\tau_{k+1}) + D(\rho^\tau_{k+1}) = J_k(\rho^\tau_{k+1}) \leq J_k(\rho^\tau_{k}) \\
   &= E(\rho^\tau_{k}) + D(\rho^\tau_{k}) \leq J_{k-1}(\rho^\tau_{k}). 
 \end{align}
 From this we extract
 \[\frac{1}{\tau}\mathcal{T}_c(\rho^\tau_{k+1},\rho^\tau_{k}) \leq E(\rho^\tau_{k}) + D(\rho^\tau_{k}) - E(\rho^\tau_{k+1}) - D(\rho^\tau_{k+1}),  \]
 which we sum to obtain
 \begin{equation}
   \label{eq:ub-tel}
   \frac{1}{\tau}\sum_{j=0}^{k-1}\mathcal{T}_c(\rho^\tau_{j+1},\rho^\tau_{j}) = E(\rho^\tau_{0}) + D(\rho^\tau_{0}) - E(\rho^\tau_{k}) - D(\rho^\tau_{k}). 
 \end{equation}

 To obtain \eqref{eq:ub-mom-bound} we pass through the Wasserstein distance so that, at the expense of a constant, we may use the triangle inequality. By \eqref{eq:moment-comparison}, we have
 \begin{align*}
   M(\rho^\tau_k) &\leq \frac{2}{\lambda}\mathcal{T}_c(\rho^\tau_k,\rho^\tau_0) + 2M(\rho^\tau_0)\\
            &\leq \frac{2\Lambda}{\lambda}\mathcal{W}_2^2(\rho^\tau_k,\rho^\tau_0) + 2M(\rho^\tau_0)\\
            &\leq k\frac{2\Lambda}{\lambda}\sum_{j=0}^{k-1}\mathcal{W}_2^2(\rho^\tau_{j+1},\rho^\tau_j)+2M(\rho^\tau_0)\\
   & \leq k\frac{2\Lambda}{\lambda^2}\sum_{j=0}^{k-1}\mathcal{T}_c(\rho^\tau_{j+1},\rho^\tau_j)+2M(\rho^\tau_0).
 \end{align*}
 Employing \eqref{eq:ub-tel} we have
 \[ M(\rho^\tau_k) \leq (\tau k) \frac{2\Lambda}{\lambda^2}\left(E(\rho^\tau_{0}) + D(\rho^\tau_{0}) - E(\rho^\tau_{k})\right)+2M(\rho^\tau_0). \]
 We've used $-D(\rho^\tau_{k+1}) \leq 0$. Next, $\tau k \leq T$ and \eqref{eq:ent-lower} with $\epsilon = \lambda^2/(4\Lambda T)$ gives
 \[ M(\rho^\tau_k) \leq C(\rho_0,K,\lambda,\Lambda,T) + \frac{1}{2} M(\rho^\tau_k),  \]
 which is \eqref{eq:ub-mom-bound}.  Finally, for \eqref{eq:ub-pos-bound} note
 \[ \int (\rho^\tau_k\log \rho^\tau_k)_{+} \dVolg  \leq J_{k-1}(\rho^\tau_k) + \int (\rho^\tau_k\log \rho^\tau_k)_{-} \dVolg.\] Estimates for $J_{k-1}(\rho^\tau_k)$ are straight forward since this quantity is nonincreasing. In addition \eqref{eq:ent-lower} combined with \eqref{eq:ub-mom-bound} yields \eqref{eq:ub-pos-bound}. We note also that the entropy and drift permit straightforward estimates (as required by Lemma \ref{lem:glob-heat-eqn}) once \eqref{eq:ub-mom-bound} and \eqref{eq:ub-pos-bound} are obtained (see \cite{JKO98} for details). 

 We integrate \eqref{eq:ub-mom-bound} and \eqref{eq:ub-pos-bound} to obtain there is a curve $\rho(t)$ such that as $\tau \rightarrow 0$, at least up to a subsequence  $\rho^\tau(t) \rightarrow \rho(t)$ weakly on $[0,T ]\times M$ and by diagonalization we obtain a weak/narrow limit on $[0,\infty)\times M$.

 Having obtained weak convergence and equation \eqref{eq:bw-ub-fd} the conclusions of Theorem \ref{thm:main} Case (ii) follow essentially verbatim the proof of Theorem Case (i) (from the point equation \eqref{eq:heat-eq-deriv1}). Note Lemma \ref{lem:glob-heat-eqn}, which is where the condition $\Vert \nabla \psi \Vert \leq C(1+\psi)$ is employed, yields the corresponding distributional form of the heat equation. 
\end{proof}

\section{The Bregman case} \label{sec:Bregman}
In this section we specialize to the Bregman case (Theorem \ref{thm:main2}). Let $\varphi$ be a smooth function defined on an open convex subset $N$ of $\mathbf{R}^n$, such that $D^2 \varphi(x)$ is positive definite for $x \in N$. In particular, $\varphi$ is strictly convex. As in \eqref{eq:breg-def}, let the cost function $c(x, y)$ be the Bregman divergence:
\[
c(x, y) = B_{\varphi}(x, y) = \varphi(x) - \varphi(y) - D\varphi(y) \cdot (x - y).
\]
A direct computation (under the Euclidean coordinate system) gives
\begin{equation} \label{eqn:Bregman.cost.derivative}
c_{x^i}(x, y) = D_i\varphi(x)  - D_i\varphi(y), \quad c_{y^i}(x, y) = \varphi_{y^i y^k}(y) (x^k - y^k),
\end{equation}
and
\begin{equation} \label{eqn:Bregman.cost.derivative2}
c_{x^i, y^j}(x, y) = - D_{ij} \varphi(y).
\end{equation}
Setting $x = y$ in the last equation shows that $c$ induces, in the sense of \eqref{eq:cg-def}, the Hessian metric corresponding to the convex potential $\varphi$.\footnote{More generally, a Hessian manifold is a Riemannian manifold whose metric can be locally expressed in the form \eqref{eqn:Bregman.cost.derivative2}. See \cite{shima2007geometry} for an in-depth study. Here, we only consider the case where \eqref{eqn:Bregman.cost.derivative2} holds globally, as otherwise the Bregman divergence is not globally defined.} On the other hand, writing 
\begin{equation} \label{eqn:Bregman.mixed}
\begin{split}
c(x, y) &= \frac{1}{2}|x - D\varphi(y)|^2 + \left( \varphi(x) - \frac{1}{2}|x|^2\right) \\
&\quad - \left(\varphi(y) + \frac{1}{2}|D\varphi(y)|^2 - y \cdot D\varphi(y)  \right)
\end{split}
\end{equation}
shows that the Bregman cost is, up to linear terms, Euclidean in the mixed coordinate system $(x, D\varphi(y))$ on $N \times N$. See \cite{RW23}, and the references therein, for the properties and applications of the Bregman--Wasserstein divergence $\mathcal{B}_{\varphi}$ defined by \eqref{eq:bw-div}. 

Since $D^2 \varphi$ is positive definite, $D\varphi$ is a diffeomorphism.\footnote{See \cite{R70} for standard results in convex analysis.} It follows from \eqref{eqn:Bregman.cost.derivative} and \eqref{eqn:Bregman.cost.derivative2} that $c$ satisfies conditions A1 and A2. From \eqref{eqn:Bregman.cost.derivative}, we see that $y_t$ is a $c$-segment from $y_0$ to $y_1$ if and only if
\[
D \varphi(y_t) = (1 - t) D \varphi(y_0) + t D \varphi(y_1).
\]
Thus $N$ is $c$-convex with respect to itself if and only if $D \varphi(N)$ is convex. Note that $D \varphi(N)$ is convex if $(N, \varphi)$ is a convex function of Legendre type. Unfortunately, $(N, \varphi)$ being of Legendre type is not sufficient to yield the other conditions required in Theorem \ref{thm:main2} which employs Theorem \ref{thm:main}.

\begin{proof}[Proof of Theorem \ref{thm:main2}]
From the above discussion we see that under the assumptions of the theorem, the Bregman cost satisfies A1 and A2, and $N$ is $c$-convex with respect to itself. 

Consider Case (i). Since $M$ is precompact, there exists constants $c_0, c_1 > 0$ such that $c_0 I \leq D^2 \varphi \leq c_1 I$. From this it is easy to see that
\[
c(x, y) = B_{\varphi}(x, y) \asymp |x - y|^2 \asymp d^2(x, y), \quad x, y \in M.
\]
Thus Theorem \ref{thm:main}(i) applies. Clearly Case (ii) is a special case of Theorem \ref{thm:main}(ii).
\end{proof}

\begin{example}
Let $M = \mathbf{R}^n$ and let $\varphi$ be a smooth convex function on $\mathbf{R}^n$ such that $c_0I \leq D^2 \varphi \leq c_1I$ for some $c_0, c_1 > 0$. Let $g$ be the Hessian metric $g = D^2 \varphi$. From standard comparison results, we see that $(M, g)$ satisfies the assumptions of Theorem \ref{thm:main2}(ii). Apart from the Bregman divergence, another explicit cost function is the (modified) Mahalanobis distance defined by
\[
c(x, y) = \frac{1}{2}(x - y)^{\top} D^2 \varphi(y) (x - y).
\]
Nevertheless, we note that Brenier's theorem does not apply to this cost. Also see \cite{HCTM20} for the special case where $D^2 \varphi$ is diagonal.
\end{example}

\begin{example}[Fisher-Rao geometry of multivariate Gaussian distributions]
For $n \geq 1$, let $\mathcal{S}_+^n$ be the space of $n \times n$ positive definite matrices. Let $N = \mathbf{R}^n \times \mathcal{S}_+^n$ which we identify with the space of multivariate Gaussian distributions on $\mathbf{R}^n$ (i.e., $(m, \Sigma) \leftrightarrow \mathcal{N}(m, \Sigma)$).\footnote{The set $N$ is relatively open so can be identified with an open convex set of some Euclidean space.} Consider the metric $g$ whose line element is given by
\[
\dd s^2 = (\dd m)^{\top} \Sigma^{-1} (\dd m) + \frac{1}{2} \mathrm{tr} \left( (\Sigma^{-1} \dd \Sigma )^2\right).
\]
It is the Fisher--Rao information metric on the space of Gaussian distributions \cite{calvo1991explicit}. The authors of \cite{calvo1991explicit} derived explicit solutions to the Riemannian geodesic equation, but analytic formulas for the Riemannian distance (which requires integrating along the geodesic) are only available in special cases. 

Consider the alternative parametrization
\[
\theta = (\theta_1, \theta_2) = \left( \Sigma^{-1} \mu, \frac{-1}{2} \Sigma^{-1} \right),
\]
which takes values in $\mathbf{R}^n \times (-\mathcal{S}_+^n)$.\footnote{This is the canonical parameter under which the normal distribution is an exponential family \cite{A16}.} Consider
\[
\varphi(\theta) = -\frac{1}{4} \theta_1^{\top} \theta_2^{-1} \theta_1 - \frac{1}{4} \log \det (-2 \theta_2),
\]
which corresponds to the log-partition function of $N(m, \Sigma)$. Then $\varphi$ is smooth and convex, and its Bregman divergence (which can be expressed in closed form) induces $g$. Thus Theorem \ref{thm:main2}(i) can be applied on compact subsets of $N$. Analogous statements can be made on any regular exponential family \cite{B14} on which a convex potential induces the Fisher-Rao metric.
\end{example}

\section{Conclusion and future directions} \label{sec:conclusion}
In this paper we propose and prove the convergence of a modified JKO scheme for the Fokker-Planck equation on a Riemannian manifold, where the (squared) Wasserstein distance is replaced by an optimal transport cost whose cost function is compatible with the underlying metric. This is not only of theoretical interest but also allows us to use cost functions, and hence transport costs, which are computationally more tractable than the Riemannian distance and Wasserstein distance. As a specific example, we consider the Bregman divergence which is a tractable cost on a Hessian Riemannian manifold induced by a convex potential. The following directions arise naturally:  
\begin{itemize}
\item[(i)] Computational algorithms for implementing the modified JKO scheme. In the Euclidean case, one way to implement the JKO step \eqref{eq:jko-def} is to use  input convex neural networks \cite{amos2017input} (which approximate the Brenier potential), possibly with entropic regularization (see \cite{bunne2022proximal, makkuva2020optimal, mokrov2021large} and the references therein). Since the optimal transport for a Bregman cost becomes Euclidean after a coordinate transformation, we expect that similar techniques can be applied to implement Wasserstein gradient flows of measures over a Hessian manifold. (We note that an implementation for a specific diagonal metric is given in \cite{HCTM20}.) 
\item[(ii)] Investigation of other tractable costs. One of the main advantages of the Euclidean quadratic cost $|x - y|^2$ is Brenier's theorem which characterizes the optimal transport map as a convex gradient. For general costs, when the optimal transport map exists it is given as a $c$-gradient of a suitable $c$-convex potential. Whilst $c$-convexity is generally abstract and not easy to work with, for some cost functions $c$-convexity can be expressed in terms of ordinary convexity. For example, consider the Dirichlet optimal transport problem \cite{PW18b} on the open unit simplex $\Delta_n = \{ p = (0, 1)^n : \sum_i p^i = 1\}$ with cost function
\begin{equation} \label{eqn:log.cost}
c(p, q) = \log \left( \sum_{i = 1}^n \frac{1}{n} \frac{q^i}{p^i} \right) - \sum_{i = 1}^n \frac{1}{n} \log \frac{q^i}{p^i}.
\end{equation}
It can be shown that the Kantorovich potential can be expressed in terms of a concave function on $\Delta_n$ and hence the techniques in (i) and \cite{RW23} may apply.\footnote{Theorem \ref{thm:main} implies that the corresponding JKO-scheme converges to the Riemannian Fokker--Planck where the metric is induced by \eqref{eqn:log.cost}. This answers (negatively) the conjecture in \cite[Section 1.1]{PW18b} that the gradient flow is nonlocal.} Also see \cite{WZ22} which generalizes \eqref{eqn:log.cost} to a one-parameter family of logarithmic costs.
\item[(iii)] Relaxing technical conditions and extending to other PDEs. To establish convergence of the JKO scheme we impose several technical conditions. Some of these, including the two-sided comparison between $c(x, y)$ and $d^2(x, y)$ as well as A2, are restrictive,\footnote{We note that the former condition is also used in \cite{P24}.} so one may ask if they can be relaxed. Also, in this paper we consider only the Fokker--Planck equation which is the benchmark case of Wasserstein gradient flows. Extending the modified JKO scheme to other PDEs may stimulate applications of the theory to non-Euclidean spaces.
\item[(iv)] The effect of the A3 condition. Regularity for optimal transport costs requires a novel, but initially perplexing, condition known as A3 which was introduced by Ma, Trudinger, and Wang \cite{MTW05}. This condition was subsequently found to have numerous interesting interpretations and consequences \cite{KM10,R23,TW09}. The authors wonder if under this condition the associated smoothness of the optimal transport maps would lead to quantitative convergence rates.  
\end{itemize}

\appendix
\section{Weak formulations of the Fokker--Planck equation}
\label{sec:weak-form-heat}

We've relied on the following Lemmas concerning the weak form of the Fokker--Planck equation. The global result is taken from Zhang's work \cite{Zhang07} and the result on precompact subsets is a straight forward adaptation. 

\begin{lemma}\label{lem:compact-heat-eqn}
  Let $(N,g)$ be a smooth Riemannian manifold and $M \subset N$ be a precompact submanifold each satisfying the hypothesis of Theorem \ref{thm:main} (i). Let $\rho \in L^1_{\text{loc}}([0,\infty) \times M)$. Assume for every $\zeta \in C^\infty(M )$ satisfying the Neumann condition $\langle \nabla \zeta(x), {\bf n}(x)\rangle = 0$ on $\partial M,$  and $\eta \in C^\infty_c([0,+\infty))$ there holds
  \begin{align*}\label{lem:comp-heat-eqn}
    0= -\eta(0) \int_{M } \zeta(x) &\rho_0(x) \dV - \int_0^\infty \int_{M } \zeta(x)\partial_t\eta(t) \rho(t,x) \dV \, \dd t\\
    &-\int_0^\infty \int_{M }( \Delta \zeta(x) - \langle \nabla \psi, \nabla \zeta \rangle )\eta(t) \rho(t,x) \dV \, \dd t.
  \end{align*}
  Then $\rho$ is the unique smooth solution of the Fokker--Planck equation satisfying the initial condition $\rho(0,x) = \rho_0$ and the Neumann boundary condition. 
\end{lemma}

\begin{lemma}\label{lem:glob-heat-eqn}
Assume $(M,g)$ is a complete Riemannian manifold with Ricci curvature bounded below and $\psi \in C^\infty(M)$ satisfies $\Vert \nabla \psi \Vert \leq C(1+\psi)$.  Let $\rho \in L^1_{\text{loc}}([0,\infty) \times M )$ satisfy $\sup_{t}M(\rho) < \infty$ and $E(\rho),D(\rho) < \infty$. Assume for every $\zeta \in C_c^\infty(M )$ and $\eta \in C^\infty_c([0,+\infty))$ there holds
   \begin{align*}
    0= -\eta(0) \int_{M } \zeta(x) &\rho_0(x) \dV - \int_0^\infty \int_{M } \zeta(x)\partial_t\eta(t) \rho(t,x) \dV \, \dd t\\
    &-\int_0^\infty \int_{M } (\Delta \zeta(x) -\langle \nabla \psi, \nabla \zeta\rangle )\eta(t) \rho(t,x) \dV \, \dd t.
   \end{align*}
   Then $\rho$ is the unique smooth solution of the Fokker--Planck equation with initial condition $\rho_0$ and finite drift and second moment. 
 \end{lemma}

Both lemmas are standard. See, for example, the definition of weak solution in \cite[Definition 4.2]{Zhang07} as well as the bootstrap argument and uniqueness result \cite[Theorem 4.5]{Zhang07} for the global case. It is the uniqueness result that requires the condition $\Vert \nabla \psi \Vert \leq C(1+\psi)$.  The precompact case follows by similar techniques. One may also consult Friedman's book \cite[Ch. 3, Theorem 7]{Friedman1964} for the parabolic regularity theory and Figalli and Glaudo's book for more details in the Euclidean case \cite[Theorem 3.3.4]{FG21}.

\bibliographystyle{plain}
\bibliography{geometry.ref} 
\end{document}